\documentclass[british]{amsart}
\usepackage[T1]{fontenc}
\usepackage[utf8]{inputenc}
\usepackage[dvipsnames]{xcolor}
\usepackage[numbers]{natbib}
\usepackage{amstext,amsthm,amssymb}
\usepackage{xparse,mathrsfs,mathtools}
\usepackage{csquotes}
\usepackage{bookmark}
\usepackage{booktabs}
\usepackage[capitalise,nameinlink]{cleveref}
\usepackage{babel}
\usepackage[yyyymmdd,hhmmss]{datetime}
\usepackage{hyperref}
\usepackage{lmodern,microtype}

\hypersetup{
	unicode=true,%
	pdfdisplaydoctitle=true,%
	pdfpagemode=UseOutlines,%
	breaklinks=true,%
	pdfencoding=unicode,psdextra,
	pdfcreator={},
	pdfborder={0 0 0},
	hidelinks
}

\newtheorem{Theorem}{Theorem}[section]

\DeclareMathOperator{\GL}{GL}

\DeclarePairedDelimiter\parentheses{\lparen}{\rparen}

\DeclarePairedDelimiter\abs{\lvert}{\rvert}
\DeclarePairedDelimiter\norm{\lVert}{\rVert}

\DeclareMathOperator{\eOpname}{e}
\NewDocumentCommand\e{ s O{} m }{
	\IfBooleanTF{#1}{%
		\eOpname_{#2}\parentheses[\big]{#3}%
	}{\eOpname_{#2}\parentheses{#3}}%
}

\numberwithin{equation}{section}
\crefformat{equation}{#2(#1)#3}
\crefformat{enumi}{#2(#1)#3}

\usepackage[dvipsnames]{xcolor}
\usepackage{tikz}

\usetikzlibrary{calc}
\usetikzlibrary{arrows}
\usetikzlibrary{backgrounds}
\usetikzlibrary{positioning}
\usetikzlibrary{decorations.pathreplacing}

\definecolor{othercolor}{gray}{0.80}

\definecolor{othercolorTwo}{rgb}{1,1,1}

\title{Diophantine approximation with prime restriction in function fields}



\author{Stephan~Baier}
\address{Stephan~Baier\\%
	Ramakrishna Mission Vivekananda Educational Research Institute\\%
	Department of Mathematics\\%
	G.\ T.\ Road, PO~Belur Math, Howrah, West Bengal~711202\\%
	India}
\email{stephanbaier2017@gmail.com}
\urladdr{https://www.researchgate.net/profile/Stephan\_Baier2}

\author{Esrafil Ali Molla\smallskip\\
\MakeLowercase{WITH AN APPENDIX BY} Arijit Ganguly}
\address{Esrafil Ali Molla\\%
	Ramakrishna Mission Vivekananda Educational Research Institute\\%
	Department of Mathematics\\%
	G.\ T.\ Road, PO~Belur Math, Howrah, West Bengal~711202\\%
	India}
\email{esrafil.math@gmail.com}

\address{Arijit Ganguly\\ Department of Mathematics and Statistics\\
Indian Institute of Technology Kanpur\\
Kanpur-208016, India}
\email{arijit.ganguly1@gmail.com}

\begin{document}

\begin{abstract} In the thirties of the last century, I.\ M.\ Vinogradov  established uniform distribution modulo 1 of the sequence $p\alpha$ when $\alpha$ is a fixed irrational real number and $p$ runs over the primes. In particular, he showed that the inequality $||p\alpha||\le p^{-1/5+\varepsilon}$ has infinitely prime solutions $p$, where $||.||$ denotes the distance to a nearest integer. This result has subsequently been improved by many authors. The current record is due to Matom\"aki (2009) who showed the infinitude of prime solutions of the inequality $||p\alpha||\le p^{-1/3+\varepsilon}$. This exponent $1/3$ is considered the limit of the current technology. We prove function field analogues of this result for the fields $k=\mathbb{F}_q(T)$ and imaginary quadratic extensions $K$ of $k$. Essential in our method is the Dirichlet approximation theorem for function fields which is established in general form in the appendix authored by Arijit Ganguly. 
\end{abstract}

\maketitle


\tableofcontents

\section{Introduction}
Throughout this article, let $\varepsilon$ be an arbitrary but fixed positive real number. 

A fundamental theorem in Diophantine approximation is the following theorem due to Dirichlet. 

\begin{Theorem} \label{Dirich}
Given any real irrational $\alpha$, there are infinitely many pairs $(a,q)\in \mathbb{Z}\times \mathbb{N}$ of relatively prime integers such that 
\[
	\abs*{ \alpha - a/q } < q^{-2}.
\]
\end{Theorem}

Equivalently, 
\begin{equation}\label{eq:DirichletApprox}
    ||q\alpha|| < q^{-1}
\end{equation}
for infinitely positive integers $q$. This result is easy to prove using the pigeonhole principle or the continued fraction expansion of $\alpha$. The latter gives explicit solutions to the above inequalities.  

It is natural to study the solubilty of \eqref{eq:DirichletApprox} when $q$ is restricted to subsets of the positive integers. The set of primes is a particularly interesting subset. The question now becomes for which $\theta>0$ one can establish the infinitude of primes $p$ such that 
\begin{equation}\label{eq:DirichletApprox:PrimeConstraint}
	\norm{ p\alpha } < p^{-\theta+\varepsilon}.
\end{equation}
In other words, we are interested in good approximations of $\alpha$ by fractions with prime denominator. 
This problem has a long history and stimulated the development of certain important tools in analytic number theory. In the following we give a brief overview of this history. 

In the thirtees of the last century, I.\ M.\ Vinogradov \cite{vinogradov2004themethod} established the infinitude of primes $p$ satifying \eqref{eq:DirichletApprox:PrimeConstraint} with $\theta = 1/5$ by a sophisticated non-trivial treatment of trigometrical sums over primes, which also allowed him to prove the ternary Goldbach conjecture for sufficiently large odd intergers without assuming any hypothesis like the Riemann Hypothesis.  In 1978, Vaughan \cite{vaughan1978onthedistribution} improved the exponent 1/5 to 1/4, using his famous identity for the von Mangoldt function and refined Fourier analytic arguments. This has been improved by Harman \cite{harman1983on-the-distribution} to 3/10 in 1983 using his sieve method which has since become an important tool in sieve theory. There were several improvements by Jia and Harman who refined numerical calculations and the sieve method itself (see \cite{jia1993on-the-distribution}, \cite{harman1996on-the-distribu} and \cite{jia2000on-the-distribution}). New ideas enabled Heath-Brown and Jia \cite{heath-brown2002the-distribution} to come very close to the exponent $1/3$ (they established $\theta=9/28)$, and $\theta=1/3$ was eventually achieved by Matom\"aki \cite{matomaki2009the-distribution}. In particular, Heath-Brown and Jia made a connection to Kloosterman sums, and Matom\"aki utilized estimates for sums of Kloosterman sums, as established by Deshouillers and Iwaniec in their celebrated work \cite{DeIw}. This exponent 1/3 is considered to be the limit of the current technology, and it is known that it is not hard to establish it under the Generalized Riemann Hypothesis (GRH) for Dirichlet $L$-functions (see the comments in \cite{heath-brown2002the-distribution}). 

Analogues of these results for quadratic number fields have recently been obtained by Harman, the first-named author, Mazumder and Technau in several papers (see \cite{Harman2019}, \cite{BT}, \cite{BM}, \cite{BMT}) who achieved an analogue of Harman's exponent $7/22$ in this setting. Here we consider an analogue for function fields. This article is organized as follows. We first give a short proof of Matom\"aki's above-mentioned result ($\theta=1/3$) under the Generalized Riemann Hypothesis (GRH) for Dirichlet $L$-functions (see Theorem \ref{Mato} below). As already mentioned, it is known that GRH implies this exponent, but we didn't find any place in the literature where this is worked out explicitly. Then we translate this proof into a version for function fields $\mathbb{F}_q(T)$, where $q\ge 7$ is an odd prime power (see Theorem \ref{FunctionfieldsMato} below). (For the cases $q=4,5$, we obtain weaker exponents along the same lines, but for $q=2,3$, we are not able to obtain a non-trivial result.) In this setting, GRH is known to hold and therefore our result is unconditional. After having established this analogue of Matom\"aki's result for function fields $\mathbb{F}_q(T)$, we work out a generalization for imaginary quadratic extensions of $\mathbb{F}_q(T)$ (see Theorem \ref{goaltheo} below).    \\ \\
{\bf Acknowledgements.} We would like to thank the Ramakrishna Mission Vivekananda Educational and Research Insititute for an excellent work environment. The second-named author thanks UGC NET fellowship for providing financial support.  

\section{A short proof of Matom\"aki's result under GRH} 
\label{distri}
We now give a proof of the above-mentioned result of Matom\"aki under the assumption of GRH. 

\begin{Theorem}[Matom\"aki] \label{Mato}
Given any real irrational $\alpha$, there are infinitely many primes $p$ such that
$$
	||p\alpha|| < p^{-1/3+\varepsilon}.
$$
\end{Theorem}

\begin{proof} ({\it under GRH})
It suffices to show that there exists an infinite sequence of real numbers $N\ge 1$ tending to infinity such that
\begin{equation} \label{goal}
\sum\limits_{\substack{N<p\le 2N\\ ||\alpha p||\le N^{-1/3+\varepsilon}}} 1 >0.
\end{equation}
As usual in this type of problems, our strategy is as follows. Theorem \ref{Dirich}, Dirichlet's approximation theorem, ensures the existence of infinitely many natural numbers $q$ such that 
\begin{equation} \label{Diri}
\left|\alpha -\frac{a}{q}\right|\le q^{-2}
\end{equation}
for some $a\in \mathbb{Z}$ with $(a,q)=1$.
We take any of these $q$'s, choose $N=q^{\tau}$ with $\tau>0$ suitable and prove \eqref{goal} for this $N$, making use of \eqref{Diri}.

Throughout the sequel, we set 
\begin{equation} \label{delta}
\delta:=N^{-1/3+\varepsilon}.
\end{equation}
We observe that under the condition \eqref{Diri}, we have
$$
||\alpha p||\le \delta
$$
for $N<p\le 2N$, provided that
\begin{equation} \label{cond}
\frac{2N}{q^2}\le \frac{\delta}{2}
\end{equation}
and 
$$
pa\equiv b \bmod{q} \mbox{ with } (b,q)=1 \mbox{ and } 0<b\le \frac{q\delta}{2}.
$$
Hence, it suffices to show that
\begin{equation} \label{suff}
S:=\sum\limits_{\substack{0<b\le q\delta/2\\ (b,q)=1}} \sum\limits_{\substack{N<p\le 2N\\ pa \equiv b \bmod{q}}} 1 \gg 1.
\end{equation}
We detect the congruence condition above using Dirichlet characters, getting
\begin{equation}\label{S}
S= \frac{1}{\varphi(q)}\cdot \sum\limits_{\chi \bmod q} \chi(a) \sum\limits_{0<b\le q\delta/2} \overline{\chi}(b)
\sum\limits_{N<p\le 2N}  \chi(p).
\end{equation}

Assuming GRH for Dirichlet $L$-functions, we know that (see \cite[page 116]{Bru})
$$
\sum\limits_{N<p\le 2N} \chi(p) = \begin{cases}  \int\limits_{N}^{2N} \frac{dt}{\log t} + O\left(N^{1/2}\log N\right) & \mbox{ if } \chi=\chi_0,\\
O\left(N^{1/2}\log N\right) & \mbox{ if } \chi\not=\chi_0\end{cases}
$$
if $q\le N$, where $\chi_0$ is the principal character modulo $q$. Plugging this into \eqref{S}, we get
\begin{equation}\label{all}
\begin{split}
S= & \frac{1}{\varphi(q)} \cdot  \int\limits_{N}^{2N} \frac{dt}{\log t} \cdot  \sum\limits_{\substack{0<b\le q\delta/2\\
(b,q)=1}} 1 +
O\left(\frac{N^{1/2}\log N}{\varphi(q)}\cdot \sum\limits_{\chi \bmod q}
\left| \sum\limits_{0<b\le q\delta/2} \overline{\chi}(b) \right|\right).
\end{split}
\end{equation}
Using Cauchy-Schwarz, expanding the modulus square and employing orthogonality relations for Dirichlet characters,  we have
\begin{equation} \label{CS}
\begin{split}
\sum\limits_{\chi\bmod q} \left|\sum\limits_{0<b\le q\delta/2} \overline{\chi}(b) \right| \le &
\varphi(q)^{1/2} \left(\sum\limits_{\chi \bmod{q}} \left| \sum\limits_{0<b\le q\delta/2} \overline{\chi}(b)\right|^2\right)^{1/2}\\
= & \varphi(q)^{1/2} \left(\sum\limits_{\chi \bmod{q}} \ \sum\limits_{0<b_1,b_2\le q\delta/2} \overline{\chi}(b_1)\chi(b_2) \right)^{1/2}\\
= & \varphi(q)^{1/2} \left(\varphi(q)\sum\limits_{\substack{0<b_1,b_2\le q\delta/2\\ b_1\equiv b_2\bmod{q}}} 1\right)^{1/2}\\
\le & \varphi(q) (q\delta)^{1/2}
\end{split} 
\end{equation}
if $\delta< 1$ since in this case, $0<b_1,b_2\le q\delta/2$ and $b_1\equiv b_2\bmod{q}$ imply $b_1=b_2$. 

Using the relations
$$
\sum\limits_{d|n} \mu(d)=\begin{cases} 1 & \mbox{ if } n= 1,\\ 0 & \mbox{ otherwise} \end{cases}
$$
and 
$$
\sum\limits_{d|q} \frac{\mu(d)}{d}=\frac{\varphi(q)}{q}
$$
and the well-known bounds 
$$
\sum\limits_{d|q} 1 \ll_{\varepsilon} q^{\varepsilon} \quad \mbox{and} \quad \frac{q}{\varphi(q)} \ll \log\log q,
$$
we obtain the following approximation for the main term, divided by the integral $\int\limits_N^{2N} dt/\log t$ on the right-hand side of \eqref{all}.
\begin{equation} \label{main}
\begin{split}
\frac{1}{\varphi(q)} 
\sum\limits_{\substack{0<b\le q\delta/2\\ (b,q)=1}} 1
= & \frac{1}{\varphi(q)} \sum\limits_{0<b\le q\delta/2} \sum\limits_{d|(b,q)}\mu(d) \\ 
= & \frac{1}{\varphi(q)} \sum\limits_{d|q} \mu(d) \sum\limits_{\substack{0<b\le q\delta/2\\ d|b}} 1\\ = &
\frac{1}{\varphi(q)} \left(\sum\limits_{d|q} \frac{\mu(d)}{d} \cdot \frac{q\delta}{2}+O\left(\sum\limits_{d|q} 1\right)\right) \\
= & \frac{\delta}{2}+O\left(\frac{1}{q^{1-2\varepsilon}}\right).
\end{split}
\end{equation}
Combining \eqref{all}, \eqref{CS} and \eqref{main}, we get
\begin{equation*}
S =  \frac{\delta}{2} \cdot  \int\limits_{N}^{2N} \frac{dt}{\log t} +
O\left(\frac{N}{q^{1-2\varepsilon}}+(q\delta)^{1/2}N^{1/2}\log N\right),
\end{equation*}
provided that $q\le N$.

Now we set
$$
N:=\left(\frac{q}{2}\right)^{2/(4/3-\varepsilon)}.
$$
Then recalling \eqref{delta}, we have
$$
\frac{2N}{q^2}=\frac{\delta}{2}
$$
in accordance with condition \eqref{cond}. Hence, 
$$
q=\frac{2N^{1/2}}{\delta^{1/2}}=2N^{2/3-\varepsilon/2},
$$
and we obtain
\begin{equation} \label{endest}
S= \frac{1}{2N^{1/3-\varepsilon}} \cdot  \int\limits_{N}^{2N} \frac{dt}{\log t} +
O\left(N^{2/3+\varepsilon/2}\right)
\end{equation}
if $\varepsilon$ is small enough. Now for $N$ sufficiently large, the main term on the right-hand side of  \eqref{endest} supercedes the error term, 
and hence
\eqref{suff} holds. Thus, the claim is established. 
\end{proof}

We note that trivial modifications in the proof of Theorem \ref{Mato} give the sharper $(\log p)^{O(1)}$ instead of the factor $p^{\varepsilon}$, but we decided to state our results with $\varepsilon$-powers for the sake of easy readability. 

In the following, we shall work out a function field analogue of the above.

\section{An analogue for the function field $\mathbb{F}_q(T)$}
\subsection{Notation} 
The following notations and conventions are used throughout the sequel. For background material on function fields, we refer the reader to \cite{Ros} and \cite{Weil}. 
\begin{itemize}
\item Let $q=p^n$ be a prime power and $\mathbb{F}_q$ be the finite field with $q$ elements. Let $\mathbb{F}_q(T)_{\infty}$ be the completion of $\mathbb{F}_q(T)$ at $\infty$
 (i.e. $\mathbb{F}_q((1/T))$).
\item The absolute value $|\cdot |$ of $\mathbb{F}_q(T)_{\infty}$ is defined as 
\begin{align*}
\left| \sum_{i= -\infty}^n a_i T^i \right| = q^n \mbox{ if }  a_n \not= 0.
\end{align*}
\item Consider the torus $\mathcal{T}= \mathbb{F}_q(T)_\infty/\mathbb{F}_q[T]$. A metric on $T$ is given by 
\begin{align*}
||x+\mathbb{F}_q[T]||:= \inf_{x'\in x+\mathbb{F}_q[T]}|x'| \quad (x\in \mathbb{F}_q(T)_{\infty}).
\end{align*}
Note that 
$\mathcal{T}$ is a compact Hausdorff space and for all $x+\mathbb{F}_q[T]\in \mathcal{T}$, we have $||x+\mathbb{F}_q[T]||\leq 1/q$. \item In more detail, if 
$$
x=\sum\limits_{i=-\infty}^{n} a_i T^i,
$$
then 
$$
||x+\mathbb{F}_q[T]||=\left|\sum\limits_{i=-\infty}^{-1} a_iT^i \right|.
$$
The sum on the right-hand side may be viewed as the fractional part of $x$. Clearly,
$$
||x||=q^{-k},
$$
where $k$ is the largest negative integer such that $a_k\not=0$. 
\item For $x\in \mathbb{F}_q(T)_{\infty}$, we also write
$$
||x||:=||x+\mathbb{F}_q[T]||. 
$$
\item If $f,g\in \mathbb{F}_q[T]$, then we write $f\approx g$ if $f$ and $g$ are associates in $\mathbb{F}_q[T]$, and we write 
$(f,g)\approx h$ if $h$ is a polynomial of maximal degree dividing both $f$ and $g$. In particular, $(f,g)\approx 1$ means that $f$ and $g$ are
relatively prime.
\item By $\mathcal{P}$
denote the set of all monic irreducible polymials $\pi \in \mathbb{F}[T]$. 
\item For $f\in \mathbb{F}_q[T]$ let $G_{f}:=\left(\mathbb{F}_q[T]/(f)\right)^{\ast}$ be the multiplicative group of units in the quotient ring
$ \mathbb{F}_q[T]/(f)$. Let $\hat{G}_f$ be the character group of $G_f$. The Dirichlet characters $\chi$ modulo $f$ are given by
$$
\chi(n)=\begin{cases} \tilde{\chi}(n+(f)) & \mbox{ if } (n,f)\approx 1,\\ 0 & \mbox{ otherwise} \end{cases}
$$
for a suitable $\tilde{\chi}\in \hat{G}_f$.
\item We define the Euler totient function on $\mathbb{F}_q[T]$ as 
$$
\varphi(f):=\sharp G_f = \sharp \hat{G}_f.
$$
\item  We define the M\"obius function on $\mathbb{F}_q[T]$ as 
$$
\mu(f):=\begin{cases} 0 & \mbox{ if } f \mbox{ is divisible by a square of a non-unit in } \mathbb{F}_q[T],\\ (-1)^{\omega(f)} & \mbox{ otherwise,} 
\end{cases}
$$ 
where $\omega(f)$ is the number of non-associate irreducible factors of $f$.
\item As usual, for $a,b\in A$ and $f\in A\setminus \{0\}$, we write $a\equiv b \bmod{f}$ if $f|(a-b)$. 
\end{itemize}

\subsection{Preliminary results}
A version of Dirichlet's approximation theorem and estimates for character sums over primes (irreducible polynomials) are available in the function field setting. We have the following results. 

\begin{Theorem} \label{Dirifunc} Let $\alpha \in \mathbb{F}_q(T)_{\infty}\setminus \mathbb{F}_q(T)$. Then there exist infinitely many pairs $(a,f)\in \mathbb{F}_q[T]\times (\mathbb{F}_q[T]\setminus \{0\})$ such that $f$ is monic, $(a,f) \approx 1$ and 
\begin{equation} \label{Diri1}
\left|\alpha -\frac{a}{f}\right|\le |f|^{-2}.
\end{equation}
\end{Theorem} 

\begin{proof}
This is a consequence of \cite[Theorem 1.1]{GaGho} with $n=1$. Note that the authors define $|f|=e^{\deg f}$, whereas we take $|f|=q^{\deg f}$ if $f\in \mathbb{F}_q[T]$.   
\end{proof}

\begin{Theorem}
Let $\chi$ be a Dirichlet character modulo $f\in \mathbb{F}_q[T]$ and $N\in \mathbb{N}$. Then 
\begin{equation} \label{primesums}
\sum\limits_{\substack{\pi \in \mathcal{P}\\ \deg(\pi)=N}} \chi(\pi) = \begin{cases}  \frac{q^N}{N} + O\left(\frac{\deg(f)\cdot q^{N/2}}{N}\right) & \mbox{ if } \chi=\chi_0,\\
O\left(\frac{\deg(f)\cdot q^{N/2}}{N}\right) & \mbox{ if } \chi\not=\chi_0\end{cases}
\end{equation}
if $\deg f\le N$, where $\chi_0$ is the principal character modulo $f$.
\end{Theorem}

\begin{proof} This is a consequence of GRH for Dirichlet $L$-functions for $\mathbb{F}_q(T)$ and was proved in \cite[chapter 4]{Ros}. 
\end{proof}

\subsection{Diophantine approximation with prime denominator in the function field $\mathbb{F}_q(T)$}
Now we establish the following analogue of Theorem \ref{Mato} for $\mathbb{F}_q(T)$. The restriction $q\ge 7$ below comes from the inequality \eqref{qcondition} at the end of the proof of Theorem \ref{FunctionfieldsMato}, which is satisfied only if $q\ge 7$. If $q=4,5$, we get the weaker exponent 
$$
\theta(q)=1-\frac{\log q}{\log(q(q-1)/2)}
$$
in place of $1/3$ using the same arguments, yielding $\theta(4)=0.226...$ and $\theta(5)=0.301...$. If $q=2,3$, we do not get a non-trivial result.  

\begin{Theorem} \label{FunctionfieldsMato} 
Assume that $q\ge 7$ is an odd prime power. 
Let $\alpha\in \mathbb{F}_q(T)_{\infty}\setminus \mathbb{F}_q(T)$. Then we have
$$
||\alpha \pi||\le |\pi|^{-1/3+\varepsilon}
$$
for infinitely many $\pi\in \mathcal{P}$.
\end{Theorem} 

\begin{proof}
It suffices to show that there are infinitely many $N\in \mathbb{N}$ such that
\begin{equation} \label{goal1}
\sum\limits_{\substack{\pi \in \mathcal{P}\\ \deg(\pi) = N\\ ||\alpha \pi||\le q^{-(1/3-\varepsilon) N}}} 1 >0.
\end{equation}
Theorem \ref{Dirifunc} ensures the existence of infinitely many monic polynomials $f\in \mathbb{F}_q[T]$ such that 
\begin{equation} \label{Diri*}
\left|\alpha -\frac{a}{f}\right|\le |f|^{-2}
\end{equation}
for some $a\in \mathbb{F}_q[T]$ with $(a,f)\approx 1$.
We take any of these $f$'s, choose $N=N(f)\in \mathbb{N}$, where $N(f)$ is an increasing function in the degree of $f$ and prove \eqref{goal1} for this $N$, making use of \eqref{Diri*}. This establishes the claim. 

Throughout the following, we set
\begin{equation} \label{delta1}
\delta:= q^{-M} \quad \mbox{with } M:=\left\lceil\left(\frac{1}{3}-\varepsilon\right)N \right\rceil.
\end{equation}
We observe that under the condition \eqref{Diri*}, we have
$$
||\alpha \pi||\le \delta
$$
for $\pi\in \mathbb{F}_q[T]$ with $\deg(\pi)=N$, provided that
\begin{equation} \label{cond1}
q^N|f|^{-2}\le \delta
\end{equation}
and 
$$
\pi a\equiv b \bmod{f} \mbox{ for some } b\in \mathbb{F}_q[T] \mbox{ with } (b,f)\approx 1 \mbox{ and } 0<|b|\le |f|\delta.
$$
Hence, it suffices to show that
\begin{equation} \label{suff1}
S:=\sum\limits_{\substack{0<|b|\le |f|\delta\\ (b,f)\approx 1}} \sum\limits_{\substack{\pi \in \mathcal{P}\\ \deg(\pi)=N\\ \pi a \equiv b \bmod{f}}} 1 \gg 1.
\end{equation}

We detect the congruence condition above using Dirichlet characters, getting
\begin{equation}\label{S1}
S= \frac{1}{\varphi(f)}\cdot \sum\limits_{\chi \bmod f} \chi(a) \sum\limits_{0<|b|\le |f|\delta} \overline{\chi}(b)
\sum\limits_{\substack{\pi \in \mathcal{P}\\ \deg(\pi)=N}}  \chi(\pi). 
\end{equation}
Plugging \eqref{primesums} into \eqref{S1}, we get
\begin{equation}\label{all1}
\begin{split}
S= & \frac{1}{\varphi(f)} \cdot  \frac{q^N}{N}\cdot  \sum\limits_{\substack{0<|b|\le |f|\delta\\
(b,f)\approx 1}} 1 +
O\left(\frac{\deg(f)\cdot q^{N/2}}{N\varphi(f)}\cdot \sum\limits_{\chi \bmod f}
\left| \sum\limits_{0<|b|\le |f|\delta} \overline{\chi}(b) \right|\right).
\end{split}
\end{equation}
Using Cauchy-Schwarz, expanding the modulus square and employing orthogonality relations for Dirichlet characters, we have
\begin{equation} \label{CS1}
\begin{split}
\sum\limits_{\chi\bmod f} \left|\sum\limits_{0<|b|\le |f|\delta} \overline{\chi}(b) \right| \le &
\varphi(f)^{1/2} \left(\sum\limits_{\chi \bmod{f}} \left| \sum\limits_{0<|b|\le |f|\delta} \overline{\chi}(b)\right|^2\right)^{1/2}\\
= & \varphi(f)^{1/2} \left(\sum\limits_{\chi \bmod{f}} \ \sum\limits_{0<|b_1|,|b_2|\le  |f|\delta} \overline{\chi}(b_1)\chi(b_2) \right)^{1/2}\\
= & \varphi(f)^{1/2} \left(\varphi(f)\sum\limits_{\substack{0<|b_1|,|b_2|\le  |f|\delta\\ b_1\equiv b_2 \bmod{f}}} 1 \right)^{1/2}\\
\le & \varphi(f) (|f|\delta)^{1/2}
\end{split} 
\end{equation}
if $\delta< 1$ since in this case, $0<|b_1|,|b_2|\le  |f|\delta$ and $b_1\equiv b_2 \bmod{f}$ imply $b_1=b_2$. 
Recalling \eqref{delta1}, and using the relations
$$
\sum\limits_{\substack{d|f\\ d \mbox{\scriptsize \ monic}}} \mu(d)=\begin{cases} 1 & \mbox{ if } f\approx 1,\\ 0 & \mbox{ otherwise} \end{cases}
$$
and 
\begin{equation} \label{rel1}
\sum\limits_{\substack{d|f\\ d \mbox{\scriptsize \ monic}}} \frac{\mu(d)}{|d|}=\frac{\varphi(f)}{|f|}
\end{equation}
and the bounds
\begin{equation} \label{rel2}
\sum\limits_{\substack{d|b\\ d \mbox{\scriptsize \ monic}}} 1 \le 2^{\deg(b)} \quad \mbox{and} \quad \frac{1}{\varphi(f)}
\le (q-1)^{-\deg(f)},
\end{equation}
we get the following approximation for the main term on the right-hand side of \eqref{all1}.
\begin{equation} \label{main1}
\begin{split}
\frac{1}{\varphi(f)} 
\sum\limits_{\substack{0<|b|\le |f|\delta\\ (b,f)\approx 1}} 1= & \frac{1}{\varphi(f)} \sum\limits_{0<|b|\le |f|\delta} \sum\limits_{\substack{d|(b,f)
\\ d \mbox{\scriptsize\ monic}}} \mu(d) \\
= & \frac{1}{\varphi(f)} \sum\limits_{\substack{d|f\\ d \mbox{\scriptsize\ monic}}} \mu(d) \cdot \sum\limits_{\substack{0<|b|\le |f|\delta\\ d|b}} 1 \\
= &
\frac{1}{\varphi(f)} \sum\limits_{\substack{d|f\\ d \mbox{\scriptsize\ monic}\\ |d|\le q^{\deg(f)-M}}} \frac{\mu(d)}{|d|} \cdot q^{\deg(f)-M}\\
= & \frac{1}{\varphi(f)} \left(\sum\limits_{\substack{d|f\\ d \mbox{\scriptsize\ monic}}} \frac{\mu(d)}{|d|} \cdot q^{\deg(f)-M}+
O\left(2^{\deg(f)}\right)\right)\\
= & q^{-M}+ O\left(\frac{2^{\deg(f)}}{(q-1)^{\deg(f)}}\right).
\end{split}
\end{equation}
Combining \eqref{all1}, \eqref{CS1} and \eqref{main1}, we get
\begin{equation*}
S =  q^{-M} \cdot  \frac{q^N}{N}+
O\left(\frac{2^{\deg(f)}q^N}{(q-1)^{\deg(f)}N}+\frac{\deg(f)\cdot (|f|\delta)^{1/2}q^{N/2}}{N}\right).
\end{equation*}

Now we set
$$
N:=\left\lfloor \frac{2\deg(f)}{4/3-\varepsilon}\right\rfloor.
$$
Then recalling \eqref{delta1}, we see that the condition \eqref{cond1} is satisfied. Hence, 
$$
\left(\frac{2}{3}-\frac{\varepsilon}{2}\right) N\le \deg(f) <\left(\frac{2}{3}-\frac{\varepsilon}{2}\right) (N+1) ,
$$
and we obtain
\begin{equation} \label{endest1}
S=q^{-M} \cdot  \frac{q^N}{N}+O_q\left(\frac{2^{(2/3-\varepsilon/2)N}q^{N}}{(q-1)^{(2/3-\varepsilon/2)N}N}+
q^{(2/3+\varepsilon/4)N}\right),
\end{equation}
and the main term is bounded from below by
$$
q^{-M} \cdot  \frac{q^N}{N}\gg_q \frac{q^{(2/3+\varepsilon)N}}{N}.
$$
This supercedes the error term on the right-hand side of \eqref{endest1} if $N$ is sufficiently large and 
\begin{equation} \label{qcondition}
\frac{2^{2/3}q}{(q-1)^{2/3}}< q^{2/3},
\end{equation}
which is the case if $q\ge 7$.  Under these conditions, \eqref{suff1} holds, and thus the claim is established.
\end{proof}

\section{Generalization to imaginary quadratic extensions of $\mathbb{F}_q(T)$}
We would like to carry over these results to imaginary quadratic extensions of $\mathbb{F}_q(T)$, where $q=p^n$ is an odd prime power. 
Let us summarize the tools needed for $\mathbb{F}_q(T)$:
\begin{itemize}
\item Dirichlet approximation theorem for $\mathbb{F}_q(T)$.
\item Prime number theorem for $\chi(\pi)$.
\item Orthogonality relations for Dirichlet characters.
\end{itemize}
We first characterize imaginary quadratic extensions, then set up an analogue of our problem for them and finally extend the above tools in a suitable way.

\subsection{Basic properties of imaginary quadratic field extensions}
First we summarize some basic properties of imaginary quadratic extensions.
\begin{itemize}
\item First, we assume that $K$ is a general finite separable extension $K$ of $k = \mathbb{F}_q(T)$. We denote the {\bf integral closure} of $A=\mathbb{F}_q[T]$ in $K$ by ${\bf A}$. (Meaning the elements of $K$ which are roots of a monic polynomial with coefficients in $A$.) 
\item In general, ${\bf A}$ is not a PID. So we need to set up our problem 
without assuming unique factorization. (There are results in the literature about $K$'s for which ${\bf A}$ is a PID. See, for example, \cite{Mac}. However, in the imaginary quadratic case, there are just four such field extensions.)  Therefore, special care is required when defining coprimality. We understand two elements $a$ and $f$ of ${\bf A}$ as being coprime and write $(a,f)\approx 1$ when the principal ideals $(a)$ and $(f)$ generated by them are coprime in the sense that they don't share a common prime ideal divisor. We recall that $a$ has a multiplicative inverse modulo $(f)$, (i.e. there exists $b\in {\bf A}$ such that $ab\equiv 1 \bmod (f)$) if and only if $a$ and $f$ are coprime in the sense above. 
\item The M\"obius and the Euler totient functions are defined on the integral ideals in the usual way as multiplicative functions with the properties that $$
\mu(\mathfrak{p}^e)=\begin{cases} -1 & \mbox{ if } e=1\\ 0 & \mbox{ if } e>1 \end{cases}
$$
and 
$$
\varphi(\mathfrak{p}^e)= \sharp({\bf A}/\mathfrak{p}^e)^{\ast}=\mathcal{N}(\mathfrak{p})^{e-1} (\mathcal{N}(\mathfrak{p})-1) \mbox{ for } e\in \mathbb{N}
$$
for prime ideals $\mathfrak{p}$ in ${\bf A}$, 
where the norm of $\mathfrak{p}$ equals
$$
\mathcal{N}(\mathfrak{p})=\sharp ({\bf A}/\mathfrak{p}). 
$$  
\item 
Now we introduce {\bf imaginary quadratic extensions}. To this end, we assume that $q$ is odd.
Let $\alpha\in A$ be a square-free polynomial which is of odd degree or whose leading coefficient is not a square in $\mathbb{F}_q$ and set
$$
K := k\left(\sqrt{\alpha}\right).
$$
Then we say that the field extension $K:k$ is imaginary quadratic. This is equivalent to saying that $\alpha$ is not a square of an element in the completion $k_{\infty}$. (This explains the expression ``imaginary quadratic'' because of its analogy to imaginary quadratic number fields.) It is clear that $K$ is separable if $q$ is odd.
\item We know that ${\bf A}$ is equal to 
$A + A\sqrt{\alpha}$ (see \cite[page 248]{Ros}). Hence, every element in ${\bf A}$ can be written as 
$$
(a_0+b_0\sqrt{\alpha})+\cdots + (a_n+b_n\sqrt{\alpha})T^n,
$$
where $a_i,b_i\in \mathbb{F}_q$.
The quotient field of the ring ${\bf A}$ is $K$. 
\item {\bf Norm and trace} of an element $f = a + b\sqrt{\alpha}\in K$ are 
given as 
$$
\mbox{\bf Norm}(f) =  (a+b\sqrt{\alpha})(a-b\sqrt{\alpha})=a^2- b^2 \alpha\in k
$$
and 
$$
\mathbf{Tr}(f)=(a+b\sqrt{\alpha})+(a-b\sqrt{\alpha})=2a\in k.
$$
\item If $K$ is an imaginary quadratic extension of $k$, then there are only {\bf finitely many units} in ${\bf A}$. It is easy to prove this, as follows. Let $u=a+b\sqrt{\alpha}$ be a unit in ${\bf A}$. Then $\mbox{\bf Norm}(u)$ is a unit in $A$, which is equivalent to saying that $a^2-b^2\alpha=v\in \mathbb{F}_q\setminus \{0\}$.  Consider first the case when $\deg \alpha>0$. Then $b^2\alpha+v$ with $v\in \mathbb{F}_q\setminus \{0\}$ has odd degree or a leading coefficient which is not a square in $\mathbb{F}_q$ unless $b=0$. Hence, $b^2\alpha +v$ cannot be the square $a^2$ in $k$ unless $b=0$. Therefore, $a^2-b^2\alpha=v$ can be a unit in $A$ only if $b=0$. Hence, $\mbox{\bf Norm}(u)=a^2$, which is a unit in $A$ iff $a\in \mathbb{F}_q\setminus \{0\}$. We deduce that $u\in \mathbb{F}_q\setminus\{0\}$. Now consider the case when $\deg \alpha=0$ and hence $\alpha$ is a non-sqare in $\mathbb{F}_q$. Then $b^2\alpha+v$ with $v\in \mathbb{F}_q\setminus \{0\}$ has a leading coefficient which is not a square in $\mathbb{F}_q$ unless $b\in \mathbb{F}_q$ and $b^2\alpha+v$ is a square $a^2$ in $\mathbb{F}_q$. Hence,
$u=a+b\sqrt{\alpha}$ with $a,b\in \mathbb{F}_q$. In both cases, $u\in \mathbb{F}_q+\mathbb{F}_q\sqrt{\alpha}$ and thus we have only finitely many units in ${\bf A}$.  
\item The {\bf absolute value} on $k$ extends uniquely to an absolute value on $K$ which we also denote by $|.|_{\infty}$ or simply by $|.|$. (The situation is analogous to that of imaginary quadratic extensions of $\mathbb{Q}$.) For $f\in K$, we have 
$$
|f|=\sqrt{|\mbox{\bf Norm}(f)|}. 
$$    
\item Let $K_{\infty}$ be the completion of $K$ with respect to the absolute value $|.|$. Then we have 
$K_{\infty}=k_{\infty}(\sqrt{\alpha})=k_{\infty}+k_{\infty}\sqrt{\alpha}$. Thus every element in $K_{\infty}$ can be written in the form
$$
\sum\limits_{i=-\infty}^n (a_i +b_i\sqrt{\alpha})T^i
$$
with $a_i,b_i\in \mathbb{F}_q$.
\item The {\bf norm of an ideal} $\mathfrak{a}\subseteq {\bf A}$ is defined as 
$$
\mathcal{N}(\mathfrak{a})= \sharp\left({\bf A}/\mathfrak{a}\right).
$$
If $\mathfrak{a}=(f)$ with $f\in {\bf A}$, then 
$$
\mathcal{N}(\mathfrak{a}) = |\mbox{\bf Norm}(f)|.
$$
\item Throughout the sequel, we allow all $O$-constants to depend on the field $K$. 
\end{itemize}

\subsection{Setup of the problem}
To set up our analogue of the problem for an imaginary quadratic field $K$, we need to define a metric $||.||$ corresponding to that for the torus $\mathcal{T}$ in this context.  
\begin{itemize}
\item Consider now the torus $\mathbb{T}= K_\infty/{\bf A}$. A metric on $\mathbb{T}$ is given by 
\begin{align*}
||x+{\bf A}||:= \inf_{x'\in x+{\bf A}}|x'| \quad (x\in K_{\infty}).
\end{align*}
Note that 
$\mathbb{T}$ is again a compact Hausdorff space. 
\item In more detail, if 
$$
x=\sum\limits_{i=-\infty}^{n} (a_i+b_i\sqrt{\alpha}) T^i,
$$
then 
$$
||x+{\bf A}||=\left|\sum\limits_{i=-\infty}^{-1} (a_i+b_i\sqrt{\alpha})T^i \right|.
$$

\item For $x\in K_{\infty}$, we also write
$$
||x||:=||x+{\bf A}||. 
$$
\end{itemize}

Our goal is to prove the following theorem, extending our result on Diophantine approximation with primes to imaginary quadratic fields. 

\begin{Theorem} \label{goaltheo}
Let $K$ be an imaginary quadratic extension field extension of $k=\mathbb{F}_q(T)$, where $q\ge 7$ is an odd prime power. Suppose that $\alpha \in K_{\infty}\setminus K$. Then there exist infinitely many principal prime ideals $\mathfrak{p}\subset {\bf A}$ such that  
$$
||\alpha\pi || \le |\pi|^{-1/3+\varepsilon}
$$
for a suitable generator $\pi$ of $\mathfrak{p}$.
\end{Theorem}

The proof of this theorem is laid out in the remainder of this article.

\subsection{Diophantine approximation}
An essential tool in our method is the Dirichlet approximation theorem. A general version of this theorem for function fields has been established by Arijit Ganguly. This is proved in the appendix which is authored by him and of independent interest. Here we use this result to derive the following version of Dirichlet's approximation theorem for imaginary quadratic extensions of function fields. The reader may note the similarity to Theorem \ref{Dirifunc}. 

\begin{Theorem} \label{diophant} Assume that $K$ is an imaginary quadratic field extension of $k=\mathbb{F}_q(T)$. Then there exists a constant $c>0$ such that for every $x\in K_{\infty}\setminus K$, there are infinitely many elements $a/f$ of $K$ with $(a,f)\in {\bf A}\times ({\bf A}\setminus \{0\})$ such that
\begin{equation} \label{Dioappro}
\left| x -\frac{a}{f} \right|\le \frac{c}{|f|^2}.
\end{equation}
\end{Theorem}

\begin{proof}
We apply Theorem \ref{consequence} in the Appendix for the case when our field $K$ is an imaginary quadratic extension of $k$ and $S=\{\infty\}$ (the singleton consisting of the place at infinity). In this case, $\mathscr{O}_S={\bf A}$. Hence, Theorem 1.3. implies that for any given $x \in K_{\infty}$, there exist infinitely many $(a,f)\in {\bf A}\times ({\bf A}\setminus \{0\})$ such that 
$$
|fx-a|\le \frac{c}{|f|}
$$  
for some constant $c>0$. Recalling that the group of units in ${\bf A}$ is finite, it follows that if $x\in K_{\infty}\setminus K$, there are infinitely many elements $a/f$ of $K$ such that
$$
\left|x-\frac{a}{f}\right|\le \frac{c}{|f|^2}.
$$
\end{proof} 
$ $\\
{\bf Remark 1:} Without loss of generality, the pairs $(a,f)$ in Theorem \ref{diophant} can be assumed not to share a non-unit common divisor. However, if ${\bf A}$ is not a PID, then this doesn't imply that the ideals $(a)$ and $(f)$ are coprime. This will cause some alterations in our method, as compared to our treatment of $k=\mathbb{F}_q(T)$. An essential point in the said treatment was the existence of a sequence of coprime pairs $(a,f)$ satisfying the conditions in Theorem \ref{Dirifunc} such that $|f|\rightarrow \infty$. In our treatment of the field $K$, we shall use the fact that there exists a sequence of Dirichlet approximations $a/f$ satisfying the conditions in Theorem \ref{diophant} such that $\mathcal{N}(\gcd((a),(f))^{-1}(f))\rightarrow \infty$. This can be seen easily as follows: Using Theorem \ref{diophant}, there exists a sequence of Dirichlet approximations $a_n/f_n$ satisfying the conditions in this theorem such that $|f_n|<|f_{n+1}|$ for all $n\in \mathbb{N}$. Since $|.|$ takes only values $q^{M/2}$ with $M\in \mathbb{Z}_{\ge 0}$, it follows that $|f_n|\rightarrow \infty$ as $n\rightarrow \infty$. Now using \eqref{Dioappro} for $a/f=a_n/f_n,a_{n+1}/f_{n+1}$ together with the triangle inequality, we deduce that
$$
\left|\frac{a_n}{f_n}-\frac{a_{n+1}}{f_{n+1}}\right|\le \frac{2c}{|f_n|^2}
$$
and hence
$$
|a_nf_{n+1}-a_{n+1}f_n|\le \frac{2c|f_{n+1}|}{|f_n|},
$$
which upon squaring implies
$$
\mathcal{N}(\gcd((a_{n+1}),(f_{n+1})))\le \mathcal{N}((a_nf_{n+1}-a_{n+1}f_n))\le \frac{(2c)^2\mathcal{N}((f_{n+1}))}{\mathcal{N}((f_n))}.
$$
Therefore, we have 
$$
\mathcal{N}(\gcd((a_{n+1}),(f_{n+1}))^{-1}(f_{n+1}))\ge (2c)^{-2}\mathcal{N}((f_n))=(2c)^{-2}|f_n|^2,
$$
and hence, the sequence of $\mathcal{N}(\gcd((a_{n}),(f_{n}))^{-1}(f_n))$ also diverges, as claimed.

\subsection{Reduction of the problem to primes in arithmetic progressions}
By $\mathbb{P}_0$, we denote the set of non-zero principal prime ideals in ${\bf A}$. It suffices to show that there are infinitely many $N\in \mathbb{N}$ such that
\begin{equation} \label{goal2}
\sum\limits_{\substack{\mathfrak{p} \in \mathbb{P}_0\\ \mathcal{N}(\mathfrak{p})= q^N\\ (\pi)=\mathfrak{p}\\ ||\alpha \pi||\le q^{-(1/6-\varepsilon) N}}} 1 >0.
\end{equation}
Here we recall that $\mathcal{N}(\mathfrak{p})=|\mbox{\bf Norm}(\pi)|=|\pi|^2$ if 
$\pi$ generates $\mathfrak{p}$. We use Theorem \ref{diophant} to approximate $\alpha$ from Theorem \ref{goaltheo} in the form
\begin{equation} \label{Diri2}
\left| \alpha -\frac{a}{f} \right|\le \frac{c}{|f|^2}
\end{equation}
with $(a,f)\in {\bf A}\times ({\bf A}\setminus \{0\})$. For simplicity, we write
$$
\mathfrak{D}:=\gcd((a),(f)) \quad \mbox{and} \quad \mathfrak{f}:=\mathfrak{D}^{-1}(f).
$$
We observe that \eqref{Diri2} implies
\begin{equation} \label{Diri3}
\left|\alpha -\frac{a}{f}\right|\le \frac{c}{\mathcal{N}(\mathfrak{f})}.
\end{equation}
We shall choose a suitable $N=F(\mathcal{N}(\mathfrak{f}))\in \mathbb{N}$, where $F(x)$ is an increasing function on $\{q^{n}: n\in \mathbb{Z}_{\ge 0}\}$ with $F(x)\rightarrow \infty$ as $x\rightarrow \infty$ and prove \eqref{goal2} for this $N$, making use of \eqref{Diri3}. This establishes Theorem \ref{goaltheo} because we know from Remark 1 above that there is a sequence of pairs $(a,f)$ satisfying  \eqref{Diri2} such that $\mathcal{N}(\mathfrak{f})$ tends to infinity.

Throughout the following, we set
\begin{equation} \label{delta2}
\delta:= q^{-M/2} \quad \mbox{with } M:=\left\lceil\left(\frac{1}{3}-\varepsilon\right)N \right\rceil.
\end{equation}
We observe that under the condition \eqref{Diri3}, we have
$$
||\alpha \pi||\le \delta
$$
for $\pi\in {\bf A}$ with $|\pi|^2=q^{N}$ and $(\pi,f)\approx 1$, provided that
\begin{equation} \label{cond2}
cq^{N/2}\mathcal{N}(\mathfrak{f})^{-1}\le \delta
\end{equation}
and 
$$
\pi a\equiv b \bmod{f} \mbox{ for some } b\in {\bf A} \mbox{ with } \mathfrak{E}=\mathfrak{D} \mbox{ and } 0<|b|\le |f|\delta,
$$
where we write 
$$
\mathfrak{E}:=\gcd((b),(f)).
$$
Hence, to prove Theorem \ref{goaltheo}, it suffices to show that
\begin{equation} \label{suff2}
S:=\sum\limits_{\substack{0<|b|\le |f|\delta\\ \mathfrak{E}=\mathfrak{D}}} \sum\limits_{\substack{\mathfrak{p} \in \mathbb{P}_0, \ \mathfrak{p}\nmid \mathfrak{f}\\ \mathcal{N}(\mathfrak{p})=q^N\\ (\pi)=\mathfrak{p}\\ \pi a \equiv b \bmod{f}}} 1 \gg 1,
\end{equation}
if $N$ is large enough.
The sum $S$ above may be re-written using multiplicative congruences (see \cite{mult}, for example) in the form
\begin{equation} \label{newSdef}
S=\sum\limits_{\substack{0<|b|\le |f|\delta\\ \mathfrak{E}=\mathfrak{D}}} \sum\limits_{\substack{\mathfrak{p} \in \mathbb{P}_0, \ \mathfrak{p}\nmid \mathfrak{f}\\ \mathcal{N}(\mathfrak{p})=q^N\\ (\pi)=\mathfrak{p}\\ \pi a b^{-1}\equiv^{\ast} 1 \bmod{\mathfrak{f}}}} 1,
\end{equation}
where for $x,y\in K$, the multiplicative congruence $x\equiv^{\ast} y \bmod{\mathfrak{f}}$ means that the $\mathfrak{p}$-adic valuation of $(x-y)$ at any prime ideal $\mathfrak{p}$ dividing $\mathfrak{f}$ satisfies $v_{\mathfrak{p}}((x-y))\ge v_{\mathfrak{p}}(\mathfrak{f})$.  
This is easily seen by a chain of equivalences similar to the one at the end of subsection 4.9.

\subsection{Detecting primes in arithmetic progressions}
For every principal prime ideal $\mathfrak{p}$, fix a generator $\pi(\mathfrak{p})$. Then the sum in \eqref{newSdef} may be written as 
\begin{equation} \label{Srewr}
S=\sum\limits_{\substack{0<|b|\le |f|\delta\\ \mathfrak{E}=\mathfrak{D}}} \sum\limits_{\substack{\mathfrak{p} \in \mathbb{P}_0, \ \mathfrak{p}\nmid \mathfrak{f}\\ \mathcal{N}(\mathfrak{p})=q^N\\ \pi(\mathfrak{p}) a b^{-1}\sim 1 \bmod{\mathfrak{f}}}} 1, 
\end{equation}
where $c\sim 1 \bmod{\mathfrak{f}}$ means that $c$ is multiplicatively congruent to a unit modulo $\mathfrak{f}$. This condition can be picked out by using the orthogonality relation for the group $G(\mathfrak{f})$ of Hecke-Dirichlet characters $\chi$ modulo $\mathfrak{f}$. Below we explain the term ``Hecke-Dirichlet character''. 

First, we introduce the Dirichlet characters modulo $\mathfrak{f}$ in the usual way as multiplicative functions on ${\bf A}$ which arise from characters for the group $({\bf A}/\mathfrak{f})^{\ast}$. Then we consider the subgroup of those Dirichlet characters which are trivial on the units. These give rise to characters on the principal ideals: If $\chi$ is such a character, then 
$$
\tilde{\chi}((c)):=\chi(c)
$$
is well-defined because $\chi$ is constant on the generators of a given principal ideal. Finally, we extend these ideal characters multiplicatively to the group of all principal fractional ideals $\mathfrak{a}$ which are coprime to $\mathfrak{f}$. We set $\tilde{\chi}(\mathfrak{a}):=0$ if $\mathfrak{a}$ is not coprime to $\mathfrak{f}$ (meaning that $v_{\mathfrak{p}}(\mathfrak{a})=0$ for all prime ideals $\mathfrak{p}$ dividing $\mathfrak{f}$.     
Let $U$ be the group of units in ${\bf A}$. We note that 
$$
G(\mathfrak{f})\cong({\bf A}/\mathfrak{f})^{\ast}/U(\mathfrak{f}),
$$
where 
\begin{equation} \label{Ufdef}
U(\mathfrak{f})=\{a+\mathfrak{f}\in ({\bf A}/\mathfrak{f})^{\ast} : a\in U\}.
\end{equation}

Now the character group $G(\mathfrak{f})$ satisfies the orthogonality relation
\begin{equation} \label{orth}
\frac{1}{\sharp G(\mathfrak{f})}\sum\limits_{\tilde{\chi} \in G(\mathfrak{f})} \tilde{\chi}((x)) = \begin{cases} 1 & \mbox{ if } x\sim 1 \bmod{\mathfrak{f}},\\ 0 & \mbox{ otherwise} \end{cases}
\end{equation}
for every $x\in K$. Hence, we may pick out the congruence condition in the inner sum on the right-hand side of \eqref{Srewr} using \eqref{orth}. 
We still need to pick out the condition of a prime ideal to be principal. This can be done using class group characters. Let $\mathcal{C}=\mathcal{C}({\bf A})$ be the class group for ${\bf A}$, the quotient of the group $\mathcal{I}$ of fractional ideals and the group $\mathcal{P}$ of principal fractional ideals.  We know that this group is finite abelian. Set $h:=\sharp \mathcal{C}$. Let $X(\mathcal{C})$ be the group of characters of this group. If $\psi\in X(\mathcal{C})$ and $\mathfrak{a}$ is a fractional ideal, then we set 
$$
\tilde{\psi}(\mathfrak{a})=\psi(\mathfrak{a}\mathcal{P}).
$$
In this way we get a group of characters on the fractional ideals, which we denote by $\tilde{X}(\mathcal{C})$. Then the orthogonality relation for $\tilde{X}(\mathcal{C})$ gives
\begin{equation} \label{ortho3}
\frac{1}{h}\sum\limits_{\tilde{\psi} \in \tilde{X}(\mathcal{C})} \tilde{\psi}(\mathfrak{a}) = \begin{cases} 1 & \mbox{ if } \mathfrak{a}\in \mathcal{P}, \\ 0 & \mbox{ otherwise.} \end{cases}
\end{equation}
Every character $\tilde{\chi}\in G(\mathfrak{f})$ extends to precisely $h$ characters on the group $\mathcal{I}$ of fractional ideals. If $\tilde{\chi}_1$ is one such character, then all others look like
$\tilde{\psi}\tilde{\chi}_1$, where $\tilde{\psi}\in \tilde{X}(\mathcal{C})$. These characters form a group $H(\mathfrak{f})$. We denote by $h(\mathfrak{f})$ the cardinality of $H(\mathfrak{f})$. Hence, we may combine \eqref{orth} and \eqref{ortho3} to get
$$
\frac{1}{h(\mathfrak{f})} \sum\limits_{\chi\in H(\mathfrak{f})} \chi(\mathfrak{a})=\begin{cases} 1 & \mbox{ if } \mathfrak{a} \sim 1 \bmod{\mathfrak{f}},\\ 0 & \mbox{ otherwise,}
\end{cases}
$$
where by $\mathfrak{a} \sim 1 \bmod{\mathfrak{f}}$ we mean that $\mathfrak{a}$ is principal and $a\sim 1 \bmod{\mathfrak{f}}$ for any generator $a$ of $\mathfrak{a}$. Therefore,
our sum $S$ can be rewritten as
$$ 
S=\sum\limits_{\substack{0<|b|\le |f|\delta\\ \mathfrak{E}=\mathfrak{D}}} \sum\limits_{\substack{\mathfrak{p} \in \mathbb{P}\\ \mathcal{N}(\mathfrak{p})=q^N}} \frac{1}{h(\mathfrak{f})} \sum\limits_{\chi\in H(\mathfrak{f})} \chi(\mathfrak{p}(ab^{-1})),
$$
where $\mathbb{P}$ is the set of {\it all} prime ideals in ${\bf A}$. 
Now we re-arrange summations and use the multiplicativity of $\chi$ on the ideals coprime to $\mathfrak{f}$ (note that $(ab^{-1})$ is coprime to $(f)$ and hence to $\mathfrak{f}$ by the condition $\mathfrak{E}=\mathfrak{D})$. We deduce that
\begin{equation}\label{S2}
S= \frac{1}{h(\mathfrak{f})} \sum\limits_{\chi\in H(\mathfrak{f})} \sum\limits_{\substack{0<|b|\le |f|\delta\\ \mathfrak{E}=\mathfrak{D}}} \chi(ab^{-1}) \sum\limits_{\substack{\mathfrak{p} \in \mathbb{P}\\ \mathcal{N}(\mathfrak{p})=q^N}} \chi(\mathfrak{p}),
\end{equation}
where we write $\chi(x):=\chi((x))$ for simplicity. Now we need to evaluate the inner sum over prime ideals. We will get the main term contribution from the case when $\chi=\chi_0$ is the principal character. For the other case when $\chi\not=\chi_0$, we require an upper bound for the said sum. This is captured by the prime number theorem for $\chi(\mathfrak{p})$, which is well-known. However, we need to work out the precise dependencies of the $O$-terms on the modulus $\mathfrak{f}$. Therefore, we go through the details in the upcoming two sections. We shall need properties of Hecke $L$-functions for this purpose, which are covered in the next subsection.  

\subsection{Hecke $L$-functions}
The above characters $\chi\in H(\mathfrak{f})$ give rise to Hecke characters for $K$. Formally, Hecke characters for function fields are defined on divisors (see \cite{Ros}[Chapter 9]), so let us set this up precisely. For every prime ideal $\mathfrak{p}$ in ${\bf A}$, let ${\bf A}_{\mathfrak{p}}$ the localization at $\mathfrak{p}$ and let $P$ be the unique maximal ideal in the discrete valuation ring ${\bf A}_{\mathfrak{p}}$. The primes $P$ we get in this way are precisely the finite primes of $K$. If $K$ is imaginary-quadratic, then there is exactly one more prime at infinity. Let's denote it by $\infty$.  We know that
$$
\deg(\infty)=2
$$
and so 
$$
\mathcal{N}(\infty)=q^{2}
$$
in the notation of \cite{Ros}.
Let $\mathcal{F}$ be the divisor given by 
$$
\mathcal{F}=\sum\limits_{i=1}^{\omega} e_iP_i, 
$$
where 
$$
\mathfrak{f}=\prod\limits_{i=1}^{\omega} \mathfrak{p}_i^{e_i}
$$
is the prime factorization of $\mathfrak{f}$ and $P_i$ are the primes corresponding to 
$\mathfrak{p}_i$. 

Now for every finite prime $P$ define    
$$
\lambda(P):=\chi(\mathfrak{p}),
$$
where $\mathfrak{p}$ is the prime ideal corresponding to $P$. Set 
$$
\lambda(\infty):=1.
$$
Then $\lambda$ extends uniquely to a Hecke character modulo $\mathcal{F}$. The Hecke $L$-function for $\lambda$ satisfies
\begin{equation} \label{productformula}
\begin{split}
L(s,\lambda)= & \prod\limits_{P} \left(1-\lambda(P)\mathcal{N}(P)^{-s}\right)^{-1}=
\prod\limits_{\mathfrak{p}\in \mathbb{P}} \left(1-\chi(\mathfrak{p})
\mathcal{N}(\mathfrak{p})^{-s}\right)^{-1}\left(1-\mathcal{N}(\infty)^{-s}\right)^{-1}\\
= & \prod\limits_{\mathfrak{p}\in \mathbb{P}} \left(1-\chi(\mathfrak{p})
\mathcal{N}(\mathfrak{p})^{-s}\right)^{-1}\left(1-q^{-2s}\right)^{-1}
\end{split}
\end{equation}
for $\Re s>1$.

If $\lambda$ is primitive, then all relevant information about $L(\lambda,s)$ is in \cite[Theorem 9.24.A]{Ros}:

\begin{Theorem} \label{RosenHecke}
Let $\lambda$ be a primitive Hecke character with conductor $\mathcal{F}$ and suppose $\lambda$ is not trivial on $D^0(S)$ (the group of divisors of degree $0$ with support disjoint from $S$, the set of primes on which $\mathcal{F}$ is supported). Then $L(s,\lambda)$ is a polynomial in $q^{-s}$ of degree $2g-2+\deg_K\mathcal{F}$ (here $g$ is the genus of $K$). Define 
$$
\Lambda(s,\lambda):=q^{(g-1)s}\mathcal{N}(\mathcal{F})^{s/2}L(s,\lambda).
$$ 
Then 
$$
\Lambda(s,\lambda)=\varepsilon(\lambda)\Lambda(1-s,\overline{\lambda}),
$$
where $\varepsilon(\chi)$ is a complex number of absolute value 1. 
\end{Theorem}

It is also known that the Riemann Hypothesis holds for $L(s,\lambda)$ above, i.e. all zeros of $L(s,\lambda)$ have real part equal to 1/2 (see \cite{Ros}).
We note that
$$
\deg_K\mathcal{F}=\log_q \mathcal{N}(\mathfrak{f}). 
$$

For completeness, we also cover the case when $\lambda$ is not a primitive character. If $\lambda$ and hence $\chi$ is not primitive, then $\chi$ is induced by a primitive Dirichlet character $\chi'$ modulo $\mathfrak{f}'$, where $\mathfrak{f}'|\mathfrak{f}$. Let $\lambda'$ be the corresponding Hecke character defined on divisors. Then $L(s,\lambda)$ and $L(s,\lambda')$ differ just by a finite Euler product, namely we have
$$
L(s,\lambda)=\left(\prod\limits_{\substack{\mathfrak{p}|\mathfrak{f}\\ \mathfrak{p}\nmid \mathfrak{f}'}} \left(1-\chi'(\mathfrak{p})\mathcal{N}(\mathfrak{p})^{-s}\right)\right) \cdot L(s,\lambda').
$$  
If $\mathfrak{f}'={\bf A}$, then $\chi'$ is trivial and hence $\chi=\chi_0$ is the principal character. In this case,
$$
L(s,\lambda')=\zeta_K(s)
$$
is the zeta function of $K$, for which we have the following result (see \cite[Theorem 5.9.]{Ros}).

\begin{Theorem} \label{Zeta} Suppose that the genus of $K$ is $g$. Then
there is a polynomial $L_K(u) \in \mathbb{Z}[u]$ of degree $g$ such that
$$
\zeta_K(s):=\frac{L_K(q^{-s})}{(1-q^{-s})(1- q^{1-s})}.
$$
This holds for all $s$ such that $\Re( s) > 1$ and the right-hand side provides an
analytic continuation of $\zeta_K(s)$ to all of $\mathbb{C}\setminus \{0,1\}$. $\zeta_K(s)$ has simple poles at $s = 0$ and $s = 1$. Set $\xi_K(s) = q^{(g-1)s}\zeta_K(s)$. Then for all $s\not=0,1$ one has the functional equation $\xi_K(1-s) = \xi_K(s)$.
\end{Theorem}

The Riemann Hypothesis is known to hold for $\zeta_K(s)$, which implies the prime number theorem for $K$ below (see \cite[Proof of Theorem 5.12.]{Ros} - here we keep the dependency on $g$).

\begin{Theorem} \label{primenumbertheorem} Suppose that the genus of $K$ is $g$. Then
$$
\sum\limits_{\substack{P\\ \deg(P)=N}} 1 = \frac{q^N}{N}+O\left(\frac{gq^{N/2}}{N}\right).
$$
\end{Theorem}

It follows that 
\begin{equation} \label{princ}
\sum\limits_{\substack{\mathfrak{p}\in \mathbb{P}\\ \mathcal{N}(\mathfrak{p})=q^N}} \chi_0(\mathfrak{p}) 
=\frac{q^N}{N}+O\left(\omega(\mathfrak{f})+\frac{gq^{N/2}}{N}\right),
\end{equation}
where $\omega(\mathfrak{f})$ is the number of prime ideal divisors of $\mathfrak{f}$. 

In the next subsection, we work out an upper bound for 
$$
\sum\limits_{\substack{\mathfrak{p}\in \mathbb{P}\\ \mathcal{N}(\mathfrak{p})=q^N}} \chi(\mathfrak{p})
$$
if $\chi$ is not the principal character modulo $\mathfrak{f}$.  

\subsection{Character sums over primes}
We first deal with the case when $\lambda$ is a primitive character. 
 It is convenient to switch to a new variable $u= q^{-s}.$ For each primitive Hecke  character $\lambda$ with conductor $\mathcal{F}$ and $N\in \mathbb{N}$ we define the number $C_N(\lambda)$ by  
\begin{equation} \label{(1)}
 u \frac{d}{d u}(\log L(u,\lambda))= \sum_{N=1}^{\infty} C_N(\lambda) u^N.  
\end{equation}
By Theorem \ref{RosenHecke} we have
\begin{equation} \label{fact}
     L(u, \lambda)= \alpha \displaystyle \prod_{i=1}^{2g -2 + \deg_K \mathcal{F}} (1- \alpha_i(\lambda )u)
\end{equation}
for suitable $\alpha,\alpha_i(\lambda)\in \mathbb{C}$.
Taking the logarithmic derivative on both sides of \eqref{fact}, multiplying with $u$ and comparing with \eqref{(1)}, we get
$$
C_N(\lambda)= - \sum_{i=1}^{2g -2 + \deg_K \mathcal{F}} \alpha_i(\lambda)^N.
$$
By the Riemann hypothesis for function fields, it follows that
\begin{equation}\label{CNbound}  
C_N(\lambda)= O ((g+\deg_K \mathcal{F}) q^{N/2})
=   O ((g+\log_q\mathcal{N}(\mathfrak{f}))\cdot q^{N/2}).
\end{equation}
Using \eqref{productformula}, we have 
\begin{equation*}
\begin{split}
L(s, \lambda)= &\displaystyle \prod_{d=1}^{\infty} \displaystyle \prod\limits_{\substack{ \mathfrak{p}\in \mathbb{P} \\ \mathcal{N}(\mathfrak{p})= q^d \\ }}(1- \chi(\mathfrak{p})q^{-ds})^{-1}(1-q^{-2s})^{-1}\\
 = &\displaystyle \prod_{d=1}^{\infty} \displaystyle \prod\limits_{\substack{ \mathfrak{p}\in \mathbb{P} \\ \mathcal{N}(\mathfrak{p})= q^d \\ }}(1- \chi(\mathfrak{p})u^d)^{-1}(1-u^2)^{-1}.
\end{split}
\end{equation*}
Taking the logarithmic derivative of both sides, multiplying both sides by $u$ and using $\eqref{(1)}$, we find
\begin{equation} \label{CN}
\begin{split}
  C_N(\lambda)= & \sum\limits_{\substack{j,d\in \mathbb{N}, \mathfrak{p} \in \mathbb{P} \\\mathcal{N}(\mathfrak{p})= q^d \\ dj =N}} d \chi(\mathfrak{p})^j + O(1)  \\
  = & N \sum_{\mathcal{N}(\mathfrak{p})= q^N} \chi(\mathfrak{p}) + O\left(\sum\limits_{\substack{ d|N \\ d \leq N/2 }} d \sum\limits_{\substack{\mathfrak{p}\in \mathbb{P} \\ \mathcal{N}(\mathfrak{p})=q^d}} 1 \right) + O(1)\\
  = & N \sum_{\mathcal{N}(\mathfrak{p})= q^N} \chi(\mathfrak{p}) +  O(gq^{N/2}),
\end{split}
\end{equation}
where we have used Theorem \ref{primenumbertheorem} to obtain the last line. 
Combining \eqref{CNbound} and \eqref{CN}, we deduce that
$$ 
\sum\limits_{\substack{ \mathfrak{p} \in \mathbb{P} \\ \mathcal{N}(\mathfrak{p})= q^N}} \chi(\mathfrak{p})=
O\left(\frac{(g+\log_q \mathcal{N}(\mathfrak{f}))\cdot q^{N/2}}{N}\right). 
$$
If $\chi$ is non-principal, assume $\chi'$ to be a character modulo an ideal
$\mathfrak{f}'$ dividing $\mathfrak{f}$ which induces $\chi$. In this case,  along the lines above, we have
\begin{equation} \label{chigen} \begin{split}
\sum\limits_{\substack{ \mathfrak{p} \in \mathbb{P} \\ \mathcal{N}(\mathfrak{p})= q^N}} \chi(\mathfrak{p})= &\sum\limits_{\substack{ \mathfrak{p} \in \mathbb{P} \\ \mathfrak{p}\nmid \mathfrak{f}\\ \mathcal{N}(\mathfrak{p})= q^N}} \chi'(\mathfrak{p})\\
= &
O\left(\omega(\mathfrak{f})+\frac{(g+\log_q \mathcal{N}(\mathfrak{f}'))\cdot q^{N/2}}{N}\right)\\
= & O\left(\omega(\mathfrak{f})+\frac{(g+\log_q \mathcal{N}(\mathfrak{f}))\cdot q^{N/2}}{N}\right).
\end{split}
\end{equation}
Combining \eqref{princ} and \eqref{chigen}, we get 
\begin{equation} \label{chicomb}
\sum\limits_{\substack{ \mathfrak{p} \in \mathbb{P} \\ \mathcal{N}(\mathfrak{p})= q^N}} \chi(\mathfrak{p})=
\begin{cases}
\frac{q^N}{N}  + O\left(\omega(\mathfrak{f})+\frac{(g+\log_q \mathcal{N}(\mathfrak{f}))\cdot q^{N/2}}{N}\right) & \text{if $\chi = \chi_0$ } \\
 O\left(\omega(\mathfrak{f})+\frac{(g+\log_q \mathcal{N}(\mathfrak{f}))\cdot q^{N/2}}{N}\right) & \text{if $\chi \neq \chi_0$}
\end{cases} 
\end{equation}
for any $\chi\in H(\mathfrak{f})$.  

\subsection{Counting principal ideals}
Before turning to the final calculations, we still need to establish an approximation of the form
\begin{equation} \label{idealcount}
\sum\limits_{\substack{b\in {\bf A}\\ 0<\mathcal{N}((b))\le q^U\\ \mathfrak{D}|(b)}} 1 = {\bf C}_K\cdot \frac{q^U}{\mathcal{N}(\mathfrak{D})}+O(1)
\end{equation} 
if $U\in \mathbb{N}$, where ${\bf C}_K$ is a positive constant depending on the field and the implied $O$-constant may depend on $K$. 

Writing $(b)=\mathfrak{Da}$, we have
\begin{equation} \label{idealcount2}
  \sum\limits_{\substack{b\in {\bf A}\\ 0<\mathcal{N}((b))\le q^U\\ \mathfrak{D}|(b)}} 1 =  \sum\limits_{\substack{\mathfrak{a} \mbox{\scriptsize\ integral ideal}\\ \mathfrak{Da} \mbox{\scriptsize\ principal}\\ 0<\mathcal{N}(\mathfrak{a})\le q^U/\mathcal{N}(\mathfrak{D})}} 1.
\end{equation}
We pick out the condition of $\mathfrak{Da}$ being principal using the orthogonality relation for class group characters, getting
\begin{equation} \label{idealcount2.1}
\begin{split}
\sum\limits_{\substack{\mathfrak{a} \mbox{\scriptsize\ integral ideal}\\ \mathfrak{Da} \mbox{\scriptsize\ principal}\\ 0<\mathcal{N}(\mathfrak{a})\le q^U/\mathcal{N}(\mathfrak{D})}} 1  = & \frac{1}{h} \cdot  \sum\limits_{\psi \in X(\mathcal{C})} \sum\limits_{\substack{\mathfrak{a} \mbox{\scriptsize\ integral ideal}\\ 0<\mathcal{N}(\mathfrak{a})\le q^U/\mathcal{N}(\mathfrak{D})}} \psi(\mathfrak{Da})\\
   = & \frac{1}{h} \cdot \sum\limits_{\psi \in X(\mathcal{C})} \psi(\mathfrak{D}) \sum\limits_{0 < n \le M} a_n(\psi),
\end{split}
\end{equation} 
where 
$$
q^M:= \frac{q^U}{\mathcal{N}(\mathfrak{D})} \quad \mbox{and} \quad 
a_n(\psi):= \sum\limits_{ \mathcal{N}(\mathfrak{a})= q^n }\psi(\mathfrak{a}). 
$$
The generating series for $a_n(\psi)$ is of the form
\begin{equation*}
        \sum\limits_{n=0}^{\infty}a_n(\psi) q^{-ns}= \sum\limits_{ \mathfrak{a}} \psi(\mathfrak{a}) \mathcal{N}(\mathfrak{a})^{-s} = L(\psi,s)(1-q^{-2s}).
\end{equation*}
By Theorem \ref{RosenHecke}, if $\psi$ is not the principal character, then  $L(\psi,s)(1-q^{-2s})$ is a polynomial of degree $2g$ in $q^{-s}$. By comparison of coefficients, it follows that $a_n(\psi)=0$ whenever $n>2g$ and therefore
\begin{equation} \label{suman1}
   \sum\limits_{ 0 < n \le M} a_n(\psi) = O(1)
\end{equation}
for every $M\in \mathbb{N}$, where the implied $O$-constant above depends only on $K$.

If $\psi= \psi_0$ is the principal character, then using Theorem \ref{Zeta} and writing $u:=q^{-s}$, we have 
\begin{equation*}
\begin{split}
    & \sum\limits_{n=0}^{\infty}a_n(\psi_0) u^n=  \sum\limits_{n=0}^{\infty}a_n(\psi_0) q^{-ns} =  L(\psi_0, s)(1-q^{-2s}) = 
    \zeta_K(s)(1-q^{-2s})\\ = & \frac{L_K(q^{-s})(1-q^{-2s})}{(1-q^{-s})(1- q^{1-s})} = \frac{L_K(u)(1+u)}{1-qu} = \frac{G_K(u)}{1-qu}, 
\end{split}
\end{equation*}
where $G_K(u)= L_K(u)(1+u)$ is a polynomial of degree $2g-1$. We write the fraction $G_K(u)/(1-qu)$ above as a power series. Then by comparison of coefficients, the $n$-th coefficient of this power series equals $a_n(\psi_0)$. Suppose that 
$$ 
G_K(u)= \sum\limits_{i=0}^{2g-1} c_i u^i.
$$
Then writing 
$$
\frac{1}{1-qu}=\sum\limits_{j=0}^{\infty} q^ju^j \quad \mbox{for } |u|<1/q, 
$$
we see that the $n-$th coefficient of the said power series equals
$$ 
\sum\limits_{i=0}^{\min (n, 2g-1) } c_i q^{n-i} = a_n(\psi_0).
$$
In particular, if $n\ge 2g-1$, then 
$$
a_n(\psi_0) = q^n \sum\limits_{i=0}^{2g-1} c_i q^{-i} = q^n G(q^{-1}).
$$
It follows that for all $M\in \mathbb{N}$, we have
\begin{equation}\label{suman2}
\begin{split}
    \sum\limits_{ 0 < n \le M} a_n(\psi) =& G(q^{-1}) \sum\limits_{0 < n \le M} q^n + O(1)\\ = & G(q^{-1}) \cdot  \frac{q}{q-1} \cdot  q^M + O(1) \\
    = & G(q^{-1}) \cdot  \frac{q}{q-1} \cdot \frac{q^U}{\mathcal{N}(\mathfrak{D})}  + O (1),
\end{split}
\end{equation}
where the implied $O$-constants depend on $K$. Combining \eqref{idealcount2}, \eqref{idealcount2.1}, \eqref{suman1} and \eqref{suman2}, we obtain \eqref{idealcount} with
$$
C_K:= \frac{1}{h} \cdot  G(q^{-1}) \cdot \frac{q}{q-1}. 
$$

\subsection{Final calculations}
Plugging \eqref{chicomb} into \eqref{S2}, we get
\begin{equation} \label{thefinalterm}
\begin{split}
S= & \frac{1}{h(\mathfrak{f})} \cdot  \frac{q^N}{N}\cdot  \sum\limits_{\substack{0<|b|\le |f|\delta\\
\mathfrak{E}=\mathfrak{D}}} 1 +\\ & 
O\left(\frac{\omega(\mathfrak{f})N+(g+(\log_q \mathcal{N}(\mathfrak{f}))\cdot q^{N/2})}{ h(\mathfrak{f}) N}\cdot \sum\limits_{\chi \in H(\mathfrak{f})}
\left| \sum\limits_{\substack{0<|b|\le |f|\delta\\ \mathfrak{E}=\mathfrak{D}}} \chi(ab^{-1}) \right|\right).
\end{split}
\end{equation}
We begin with estimating the character sum in the $O$-term.

Using Cauchy-Schwarz, we have
\begin{equation*}
\sum\limits_{\chi \in H(\mathfrak{f})} \left|\sum\limits_{\substack{0<|b|\le |f|\delta\\
\mathfrak{E}=\mathfrak{D}}} \chi(ab^{-1}) \right| \le 
h(\mathfrak{f})^{1/2} \left(\sum\limits_{\chi \in H(\mathfrak{f})} \left| \sum\limits_{\substack{0<|b|\le |f|\delta\\ \mathfrak{E}=\mathfrak{D}}} \chi(ab^{-1})\right|^2\right)^{1/2}.
\end{equation*}
We re-write the term on the right-hand side as 
$$
h(\mathfrak{f})^{1/2} \left(\sum\limits_{\psi \in X(\mathcal{C})} \sum\limits_{\chi \in G(\mathfrak{f})} \left| \sum\limits_{\substack{0<|b|\le  |f|\delta\\ \mathfrak{E}=\mathfrak{D}}} \psi \chi(ab^{-1})\right|^2 \right)^{1/2}.
$$
Since the class group characters $\psi$ are trivial on the principal ideals, the above equals
$$
h(\mathfrak{f})^{1/2}h^{1/2} \left(\sum\limits_{\chi \in G(\mathfrak{f})} \left| \sum\limits_{\substack{0<|b|\le  |f|\delta\\ \mathfrak{E}=\mathfrak{D}}}  \chi(ab^{-1})\right|^2 \right)^{1/2}.
$$
Expanding the modulus square and using the orthogonality relation for the character group $G(\mathfrak{f})$, the above equals
\begin{equation} \label{interm}
\begin{split}
& h(\mathfrak{f})^{1/2}h^{1/2} \left(\sum\limits_{\chi \in G(\mathfrak{f})}  \sum\limits_{\substack{0<|b_1|,|b_2|\le  |f|\delta\\ \mathfrak{E}_1=\mathfrak{D}=\mathfrak{E}_2}} \chi(b_2b_1^{-1}) \right)^{1/2}\\
= & h(\mathfrak{f})^{1/2} h^{1/2} (\sharp G(\mathfrak{f}))^{1/2} \left( \sum\limits_{\substack{0<|b_1|,|b_2|\le  |f|\delta \\ \mathfrak{E}_1=\mathfrak{D}=\mathfrak{E}_2\\ b_2b_1^{-1} \sim 1 \bmod{\mathfrak{f}}}} 1 \right)^{1/2},
\end{split}
\end{equation}
where $\mathfrak{E_i}:= \gcd ((b_i),(f))$ for $i=1,2$ and we recall that $b_2b_1^{-1} \sim 1\bmod{\mathfrak{f}}$ means that $b_2b_1^{-1}$ is multiplicatively congruent to a unit modulo $\mathfrak{f}$. We claim the following:\\ \\
{\bf Claim:} If $\delta< 1$, $0<|b_1|,|b_2|\le  |f|\delta$,  $\mathfrak{E}_1=\mathfrak{D}=\mathfrak{E}_2$ and $b_2b_1^{-1} \equiv^{\ast} 1 \bmod{\mathfrak{f}}$, then necessarily $b_1=b_2$. \\ \\
We prove this claim at the end of this subsection and proceed with our calculation. Recall the definition of $U(\mathfrak{f})$ in \eqref{Ufdef}. By the above claim, the last line in \eqref{interm} above is bounded by 
\begin{equation} \label{firststep}
\begin{split}
\le  & h(\mathfrak{f})^{1/2} h^{1/2} (\sharp G(\mathfrak{f}))^{1/2} (\sharp U(\mathfrak{f}))^{1/2} \left( \sum\limits_{\substack{0<|b|\le  |f|\delta \\ \mathfrak{E}=\mathfrak{D}}} 1 \right)^{1/2}\\
\le & h(\mathfrak{f})^{1/2} h^{1/2} (\sharp G(\mathfrak{f}))^{1/2} (\sharp U(\mathfrak{f}))^{1/2} \left( \sum\limits_{\substack{0<\mathcal{N}((b))\le  \mathcal{N}((f))\delta^2 \\  \mathfrak{D}|(b)}} 1 \right)^{1/2}\\
\ll & h(\mathfrak{f})^{1/2} h^{1/2} (\sharp G(\mathfrak{f}))^{1/2} (\sharp U(\mathfrak{f}))^{1/2} \left(\mathcal{N}(\mathfrak{D})^{-1}\mathcal{N}((f))\delta^2\right)^{1/2}\\
= & h(\mathfrak{f})^{1/2} h^{1/2} (\sharp G(\mathfrak{f}))^{1/2} (\sharp U(\mathfrak{f}))^{1/2} \mathcal{N}(\mathfrak{f})^{1/2} \delta,
\end{split}
\end{equation}
where we use \eqref{idealcount}. Here the implied $\ll$-constant depends on $K$.
Since 
\begin{equation} \label{hphirel}
h(\mathfrak{f})=h \cdot \sharp G(\mathfrak{f}) = h\cdot \frac{\varphi(\mathfrak{f})}{\sharp U(\mathfrak{f})},
\end{equation}
the last line in \eqref{firststep} above is 
$$
\le h\varphi(\mathfrak{f})\mathcal{N}(\mathfrak{f})^{1/2} \delta
$$
and hence, we obtain the final bound
\begin{equation} \label{charsumesti}
\sum\limits_{\chi \in H(\mathfrak{f})} \left|\sum\limits_{\substack{0<|b|\le |f|\delta\\
\mathfrak{E}=\mathfrak{D}}} \chi(ab^{-1}) \right| \ll_K
\varphi(\mathfrak{f})\mathcal{N}(\mathfrak{f})^{1/2} \delta.
\end{equation}

Turning to the main term on the right-hand side of \eqref{thefinalterm}, we begin with writing
$$
\frac{1}{h(\mathfrak{f})} \cdot  \frac{q^N}{N}\cdot  \sum\limits_{\substack{0<|b|\le |f|\delta\\
\mathfrak{E}=\mathfrak{D}}} 1= \frac{1}{h(\mathfrak{f})} \cdot  \frac{q^N}{N}\cdot  \sum\limits_{\substack{0<\mathcal{N}((b))\le \mathcal{N}((f))\delta^2\\
\mathfrak{D}|(b)\\ \mbox{\rm \scriptsize gcd}(\mathfrak{D}^{-1}(b),\mathfrak{f})=1}} 1.
$$
Using the relation
$$
\sum\limits_{\mathfrak{a}|\mathfrak{f}} \mu(\mathfrak{a})=\begin{cases} 1 & \mbox{ if } \mathfrak{f}=(1),\\ 0 & \mbox{ otherwise,} \end{cases}
$$
the above may be re-written in the form
\begin{equation} \label{theterm}
\frac{1}{h(\mathfrak{f})} \cdot  \frac{q^N}{N}\cdot  \sum\limits_{\mathfrak{a}|\mathfrak{f}} \mu(\mathfrak{a}) \sum\limits_{\substack{0<\mathcal{N}((b))\le \mathcal{N}((f))\delta^2\\
\mathfrak{D}\mathfrak{a}|(b)}} 1.
\end{equation}
We recall that $\delta^2=q^{-M}$ with $M\in \mathbb{N}$. In this case, using \eqref{idealcount}, we have  
\begin{equation*}
\sum\limits_{\substack{0<\mathcal{N}((b))\le \mathcal{N}((f))\delta^2\\
\mathfrak{D}\mathfrak{a}|(b)}} 1 =  {\bf C}_K\cdot \frac{\mathcal{N}((f))\delta^2}{\mathcal{N}(\mathfrak{D}\mathfrak{a})} + O(1)
=  {\bf C}_K\cdot \frac{\mathcal{N}(\mathfrak{f})\delta^2}{\mathcal{N}(\mathfrak{a})} +O(1),
\end{equation*} 
where the implied $O$-constants above may depend on $K$.
Hence, \eqref{theterm} can be written as
\begin{equation} \label{furthersimply}
 \frac{1}{h(\mathfrak{f})} \cdot  \frac{q^N}{N}\cdot  \left({\bf C}_K\sum\limits_{\mathfrak{a}|\mathfrak{f}} \frac{\mu(\mathfrak{a})}{\mathcal{N}(\mathfrak{a})} \cdot \mathcal{N}(\mathfrak{f})\delta^2+ O\left(\sum\limits_{\mathfrak{a}|\mathfrak{f}} 1\right)\right).
\end{equation}
Now we use \eqref{hphirel}, the relation
$$
\sum\limits_{\mathfrak{a}|\mathfrak{f}} \frac{\mu(\mathfrak{a})}{\mathcal{N}(\mathfrak{a})}=\frac{\varphi(\mathfrak{f})}{\mathcal{N}(\mathfrak{f})}
$$
and the bounds
$$
\sum\limits_{\mathfrak{a}|\mathfrak{f}} 1 \le 2^{\log_q \mathcal{N}(\mathfrak{f})} \quad \mbox{and} \quad \frac{1}{\varphi(\mathfrak{f})}
\le (q-1)^{-\log_q \mathcal{N}(\mathfrak{f})}
$$
corresponding to \eqref{rel1} and \eqref{rel2}, 
which are easy to deduce by looking at the prime ideal factorization of $\mathfrak{f}$. In this way, recalling $\delta^2=q^{-M}$ and \eqref{hphirel}, we arrive at the approximation
\begin{equation} \label{all2}
\frac{1}{h(\mathfrak{f})} \cdot  \frac{q^N}{N}\cdot  \sum\limits_{\substack{0<|b|\le |f|\delta\\
\mathfrak{E}=\mathfrak{D}}} 1= {\bf C}_K\cdot \frac{\sharp U(\mathfrak{f})}{h} \cdot  \frac{q^{N-M}}{N} + O\left(\frac{q^N}{N}\cdot \left(\frac{2}{q-1}\right)^{\log_q \mathcal{N}(f)}\right)
\end{equation}
for our main term.

Combining \eqref{thefinalterm}, \eqref{charsumesti}, \eqref{all2}, and again using $\delta=q^{-M/2}$ and \eqref{hphirel} and the bound
$$
\omega(\mathfrak{f})\le \log_q \mathcal{N}(\mathfrak{f}), 
$$
we get
\begin{equation*}
 S = {\bf C}_K\cdot \frac{\sharp U(\mathfrak{f})}{h} \cdot  \frac{q^{N-M}}{N} +
  O\left(\frac{q^N}{N}\cdot \left(\frac{2}{q-1}\right)^{\log_q \mathcal{N}(\mathfrak{f})}
  + \frac{q^{(N-M)/2}\log_q \mathcal{N}(\mathfrak{f})}{ N}\cdot \mathcal{N}(\mathfrak{f})^{1/2}\right).
\end{equation*}

Now we proceed similarly as at the end of section 3, where we replace $|f|$ by $\mathcal{N}(\mathfrak{f})$ and $\deg(f)$ by $\log_q \mathcal{N}(\mathfrak{f})$. Hence, we set
$$
N:=\left\lfloor \frac{2\log_q\mathcal{N}(\mathfrak{f})}{4/3-\varepsilon}\right\rfloor.
$$
Then recalling \eqref{delta2}, we see that the condition \eqref{cond2} is satisfied if $N$ is large enough. Hence, 
$$
\left(\frac{2}{3}-\frac{\varepsilon}{2}\right) N\le \log_q\mathcal{N}(\mathfrak{f}) <\left(\frac{2}{3}-\frac{\varepsilon}{2}\right) (N+1) ,
$$
and we obtain
\begin{equation} \label{endest2}
S=  {\bf C}_K\cdot \frac{\sharp U(\mathfrak{f})}{h}\cdot \frac{q^{N-M}}{N} + O_{K}\left(\frac{q^{N}}{N}\cdot \left(\frac{2}{q-1}\right)^{(2/3-\varepsilon/2)N}+
q^{(2/3+\varepsilon/4)N}\right).
\end{equation}
The main term is bounded from below by
$$
{\bf C}_K\cdot \frac{\sharp U(\mathfrak{f})}{h}\cdot \frac{q^{N-M}}{N} \gg_{K}
\frac{q^{(2/3+\varepsilon)N}}{N}.
$$
This supercedes the error term on the right-hand side of \eqref{endest2} if $N$ is sufficiently large and \eqref{qcondition} is satisfied,
which is the case if $q\ge 7$.  Under these conditions, \eqref{suff2} holds, and thus Theorem \ref{goaltheo} is established.

It remains to prove the {\bf claim} above, which is done as follows. Noting that $\mathfrak{f}$ and $\mathfrak{D}^{-1}(b_1)$ are coprime integral ideals (by $\mathfrak{E}_1=\mathfrak{D}$ and the definition of $\mathfrak{E}_1$), we have the chain of equivalences
\begin{equation*}
\begin{split}
& \ b_2\equiv b_1 \bmod{f}\\ \Longleftrightarrow &\ (f)|(b_2-b_1) \\  \Longleftrightarrow & \
\mathfrak{Df}|(b_2-b_1) \\
\Longleftrightarrow & \ \mathfrak{f}|\mathfrak{D}^{-1}(b_2-b_1)\\
\Longleftrightarrow & \ v_{\mathfrak{p}}(\mathfrak{f}) \le v_{\mathfrak{p}}(\mathfrak{D}^{-1}(b_2-b_1)) \mbox{ for all prime ideals } \mathfrak{p} \mbox{ dividing }\mathfrak{f}\\
\Longleftrightarrow & \ v_{\mathfrak{p}}(\mathfrak{f}) \le v_{\mathfrak{p}}(\mathfrak{D}^{-1}(b_2-b_1) (\mathfrak{D}^{-1}(b_1))^{-1}) \mbox{ for all prime ideals } \mathfrak{p} \mbox{ dividing }\mathfrak{f} \\
\Longleftrightarrow & \ v_{\mathfrak{p}}(\mathfrak{f})\le v_{\mathfrak{p}}((b_2b_1^{-1}-1)) \mbox{ for all prime ideals } \mathfrak{p} \mbox{ dividing }\mathfrak{f}\\
\Longleftrightarrow&  \ b_2b_1^{-1}\equiv^{\ast} 1 \bmod{\mathfrak{f}}.
\end{split}
\end{equation*}
Now if $0<|b_1|,|b_2|\le  |f|\delta<|f|$, then $|b_2-b_1|< |f|$ and hence $\mathcal{N}((b_2-b_1))<\mathcal{N}((f))$. But if $b_2\equiv b_1 \bmod{(f)}$ this implies $b_1=b_2$. Thus our {\bf claim} is proved. 

\section{Appendix by Arijit Ganguly: Dirichlet approximation in function fields}
Let $K$ be a function field of characteristic $p>1$, i.e., a finite separable extension of $\mathbb{F}_q(T)$, where $q=p^n$ for some $n\in \mathbb{N}$, and $T$ is an indeterminate. Assume that $[K:\mathbb{F}_q(T)]=d$. Denote by $M(K)$ the set of all places of $K$. Given a place $v$ of $K$, $K_{v}$ stands for the completion of $K$ with respect to $v$ which is a locally compact field. We let $\alpha_v$ denote the Haar measure on $K_v$ that has been scaled in such a way that $\alpha_v(\mathscr{O}_v)=1$, and $|\cdot|_v$ the absolute value on $K_v$ that satisfies the following for all $a\in K_v$:
\[\alpha_v(aM)=|a|_v \alpha_v(M), \text{ for any measurable } M\subseteq K_v.\]
\noindent For any $N\in \mathbb{N}$, $(K_v)^N$ is equipped with the \textit{supremum norm} defined in the natural way: $|| \overrightarrow{x}||_v:=\displaystyle \max_{1\leq i\leq N}|x_i|_v$, where $\overrightarrow{x}=(x_1,\dots,x_N)\in (K_v)^N$.\\

For $v\in M(K)$, $\mathscr{O}_v:=\{x\in K_v: |x|_v\leq 1\}$ is the maximal compact subring of $K_v$. Recall that, for any $N\in\mathbb{N}$, $(\mathscr{O}_v)^N$  is a \emph{$K_v$-lattice} (compact open $\mathscr{O}_v$ submodule) in $(K_v)^N$. Consider a finite set of places, say $S$, of $K$. The \emph{ring of $S$-integers} of $K$ is defined by the following: \[\mathscr{O}_S:=\{x\in K: x\in \mathscr{O}_v, \text{ for all } v\notin S\}.\]

We now state the \emph{Dirichlet's theorem} in this context:
\begin{Theorem}\label{main theorem}
Let $\varepsilon_v\in K_v\setminus \{0\}$ and $A_v$ be an $M\times N$ matrix over $K_v$, for any $v\in S$ and positive integers $M$ and $N$. For each $v\in S$, choose $\delta_v\in K_v\setminus \{0\}$ such that $|\delta_v|_v\geq 1$ and furthermore, \[\displaystyle \prod_{v\in S}|\delta_v|_v^N|\varepsilon_v|_v^M=q^{(M+N)(g(K)-1)+1},\] where $g(K)$ is the genus of $K$. Then there exist $\overrightarrow{x}\in (\mathscr{O}_S)^N\setminus \{\overrightarrow{0}\}$ and $\overrightarrow{y}\in (\mathscr{O}_S)^M$ satisfying the following for all $v\in S$:
\begin{equation}\label{Dirichlet}\left\{\begin{array}{rcl}||A_v\overrightarrow{x}+\overrightarrow{y}||_v
\leq |\varepsilon_v|_v, \text{ and }\\
||\overrightarrow{x}||_v\leq |\delta_v|_v\end{array}.\right.\end{equation}
\end{Theorem}
Theorem \ref{main theorem} will be proved using the adelic version of Minkowski's \emph{convex body theorem}, established in the paper \cite{Thunder}. Recall that, the \textit{ring of adeles} of $K$, denoted by $K_{\mathbb{A}}$, is defined as the set of all elements $(x_v)_{v\in M(K)}$ in $\displaystyle \prod_{v\in M(K)}K_v$ such that $|x_v|_v\leq 1$, for almost all $v\in M(K)$. One has the following diagonal embedding of $K$ inside $K_{\mathbb{A}}$: 
\[K\hookrightarrow K_{\mathbb{A}}, \alpha \mapsto (\alpha, \alpha,\dots), \, \forall \alpha \in K.\]
\noindent The \emph{idele group} of $ K_{\mathbb{A}}$, denoted by $ K_{\mathbb{A}}^\times$, is the group of all invertible elements of $ K_{\mathbb{A}}$. For every element $x=(x_v)_{v\in M(K)}\in K_{\mathbb{A}}^\times$, we will write \[|x|_{\mathbb{A}}= \displaystyle \prod_{v\in M(K)}|x_v|_v.\]
In fact, $|\cdot|_{\mathbb{A}}$ is the module on $K_{\mathbb{A}}^\times$ (see Chapter IV, \cite{Weil}). We have the following product formula (see \cite[Theorem 5, Chapter IV]{Weil}):
\begin{equation}\label{formula}
|x|_{\mathbb{A}}= \displaystyle \prod_{v\in M(K)}|x|_v=1, \text{ for all }x\in K^\times.
\end{equation}

The ring $K_{\mathbb{A}}$ is a locally compact topological ring. We denote the following Haar measure on $K_{\mathbb{A}}$ by $\alpha_{\mathbb{A}}$:
\[q^{1-g(K)}\displaystyle \prod_{v\in M(K)}\alpha_v.\] A measurable subset of the $N$-fold product $(K_{\mathbb{A}})^N$ of the ring of adeles of $K$, where $n$ is a positive integer, is said to be a \emph{star body} if it contains $\mathbf{0}$ and, for every $\mathbf{x}\in S$, $a\mathbf{x}$ lies in the interior of $S$ whenever $a\in K_{\mathbb{A}}$ is a unit in the ring $K_{\mathbb{A}}$ and satisfying $|a_v|_v\leq 1$ for all place $v\in M(K)$. The following special case of the Theorem 3 in \cite{Thunder} plays the main role in proving Theorem \ref{main theorem}:
\begin{Theorem} \label{thunder}Let $K$ be a function field and  $N\in \mathbb{N}$. Then for any $A\in \GL_N(K_{\mathbb{A}})$ with $|\det(A)|_{\mathbb{A}}< q^{N(1-g(K))}$, one has 
	\[A(K^N)\bigcap \displaystyle \prod_{v\in M(K)} (\mathscr{O}_v)^N\neq \{\vec{0}\}.\]
	\end{Theorem}
\begin{proof}[Proof of Theorem  \ref{main theorem}] For each $v\in S$, we consider the following square matrix of order $(M+N)$ with entries from the field $K_v$: 
	\[B_v:=\left(\begin{array}{rcl}\varepsilon_v^{-1}I_M & \varepsilon_v^{-1}A_v\\
	\textbf{O} & \delta_v^{-1}I_N 
	\end{array}
	\right),
	\] where $I_M$ and $I_N$ denote the identity matrices of order $M$ and $N$ respectively. For any other place $v$, we set $B_v:=I_{M+N}$. Clearly, $B:=(B_v)_{v\in M(K)}\in  \GL_N(K_{\mathbb{A}})$ and   \begin{equation}\label{det}
|\det(B)|_{\mathbb{A}}= \displaystyle \prod_{v\in S}|\delta_v|_v^{-N}|\varepsilon_v|_v^{-M}=q^{(M+N)(1-g(K))-1}.
	\end{equation}
	On the other hand, one observes that the measure of $ \left(\displaystyle \prod_{v\in M(K)} \mathscr{O}_v\right)^{M+N}$ is the following:  \begin{equation}\label{volume}\alpha_{\mathbb{A}}^{M+N}\left(\displaystyle \prod_{v\in M(K)} (\mathscr{O}_v)^{M+N}\right)=q^{(1-g(K))(M+N)}.\end{equation}From \eqref{det} and \eqref{volume}, in view of Theorem \ref{thunder}, one concludes that, there exists  $(\overrightarrow{y}, \overrightarrow{x})\in K^{M+N}\setminus \{\overrightarrow{0}\}$ such that $\left|\left|B_v\left(\begin{array}{rcl}\overrightarrow{y}\\ \overrightarrow{x}\end{array}\right)\right|\right|_v\leq 1$, for every $v\in M(K)$. Since $B_v=I_{M+N}$ for all $v\neq S$, so we have $(\overrightarrow{y}, \overrightarrow{x})\in (\mathscr{O}_S)^{M+N}$. The system of inequalities given in \eqref{Dirichlet} is thus satisfied by $(\overrightarrow{y}, \overrightarrow{x})$ for every $v\in S$.  It only remains to show that $\overrightarrow{x} \neq \overrightarrow{0}$. \\
	
	If $\overrightarrow{x}=\overrightarrow{0}$ then $\overrightarrow{y}\neq \overrightarrow{0}$.  Writing $\overrightarrow{y}=(y_1,\dots, y_M)$, where $y_1,\dots, y_M \in K$, one has $j\in \{1,\dots, M\}$ such that $y_j\neq 0$. Since $\overrightarrow{x}=\overrightarrow{0}$, it follows from \eqref{Dirichlet} that, $||\overrightarrow{y}||_v\leq |\varepsilon_v|_v<1$, for all $v\in S$. This provides us with the following:
	\[\left\{\begin{array}{rcl}|y_j|_v< 1 && \text{ if }v\in S\\ |y_j|_v\leq 1 && \text{ if }v\notin S
	\end{array}\right.,\] which in turn implies that  $|y_j|_{\mathbb{A}}=\displaystyle \prod_{v\in M(K)} |y_j|_v<1$, contradicting the product formula given by \eqref{formula}. The proof of Theorem \ref{main theorem} is hereby complete.
	\end{proof} 
	
	As an immediate corollary of Theorem \ref{main theorem}, we get the following version of Dirichlet’s theorem for function fields. See \cite{BV} for the analogous statement in number fields which in turn is deduced from \cite{Weil}. For a finite set of places $S$, we denote by $K_S$ the product of the completions $K_v$ as $v$ ranges over $S$ and $\iota_S$ is the diagonal
embedding of $K$ into $K_S$.

\begin{Theorem} \label{consequence} There exists a constant $C> 0$ depending only on $K$ and $S$, such that for every $x \in K_S$ and for every $Q>0$, there exist infinitely many $p \in \mathscr{O}_S$, $q \in \mathscr{O}_S \setminus\{0\}$ with
$$
|| \iota_S(q) \cdot x + \iota_S(p) || \le C||\iota_S(q)||^{-1}.
$$
\end{Theorem}

\end{document}